\documentclass{sinst}

\usepackage{amsbsy,amsfonts,amsmath,amssymb,amsthm}
\usepackage{array}
\usepackage{bm}
\usepackage{color}
\usepackage{enumerate}
\usepackage{mathrsfs}
\usepackage{latexsym}
\usepackage[colorlinks=true]{hyperref}
\usepackage{mathtools}
\mathtoolsset{showonlyrefs, showmanualtags}
\hypersetup{urlcolor=blue, citecolor=blue, linkcolor=blue}

\definecolor{wineRed}{rgb}{0.7,0,0.3}
\definecolor{grandBleu}{rgb}{0,0,0.8}
\definecolor{darkGreen}{rgb}{0,0.4,0}
\definecolor{blueViolet}{rgb}{0.4,0,1.0}
\definecolor{bloodOrange}{rgb}{0.85,0.05,0}
\definecolor{mycolor}{rgb}{0.8,0,0.2}

\usepackage{comment}

\DeclareMathAlphabet{\mathpzc}{OT1}{pzc}{m}{it}

\numberwithin{equation}{section}

\theoremstyle{plain}
 
\newtheorem{lemma}{Lemma}[section]

\newtheorem{theorem}{Theorem}

\theoremstyle{definition}
\newtheorem{definition}{Definition}
\newtheorem{remark}{Remark}

\newtheorem{example}{Example}

\def\N{\mathbb{N}}
\def\R{\mathbb{R}}

\DeclareMathOperator{\diver}{div}

\DeclareMathOperator*{\olim}{\overline{\lim}}
\DeclareMathOperator*{\ulim}{\underline{\lim}}

\begin{document}
\label{page:t}
\thispagestyle{plain}

\title{
Energy Dissipative Solution to a Nonlinear Parabolic Systems with Unknown Dependent Coefficients
}
\author{{\sc Naotaka Ukai}
}
\affiliation{ Division of Mathematics and Informatics, \\Graduate School of Science and Engineering, Chiba University, \\1-33, Yayoi-cho, Inage-ku, 263-8522, Chiba, Japan}
\email{24wd0101@student.gs.chiba-u.jp}
\footcomment{
\\
}
\maketitle

\noindent
{\bf Abstract.}
In this paper, we investigate a system of parabolic partial differential equations with unknown-dependent coefficients that integrates two models: an anisotropic orientation-adaptive denoising process in image processing and a phase-field model of grain-boundary motion in materials science.
In recent years, several studies have attempted to develop a unified framework for treating these two research areas by considering pseudo-parabolic systems obtained through the introduction of the energy-dissipation operator $ - \Delta \partial_t $. However, the mathematical models for image processing and grain-boundary motion are originally formulated as parabolic systems. Therefore, establishing a unified analytical framework for such parabolic models remains an open problem.
In this paper, we address this open problem by introducing a notion of solution that reproduces energy dissipation in parabolic systems, which we call an energy dissipative solution. As the main result, we clarify conditions that guarantee the existence of such solutions.
The results of this paper establish a unified analytical framework for parabolic models, which has remained unresolved, and provide a solid theoretical foundation for advanced problems spanning both image processing and materials science.

\section{Introduction}
Let $T>0$ be a fixed positive constant and let $M,N\in\mathbb{N}$. Let $\Omega\subset\mathbb{R}^N$ be a bounded domain with Lipschitz boundary $\Gamma:=\partial\Omega$ and unit outer normal $\bm n_\Gamma$. We set $Q:=(0,T)\times\Omega$ and $\Sigma:=(0,T)\times\Gamma$. 

In this paper, we consider the following system of pseudo-parabolic PDEs, denoted by $(\mathrm{S})$:\vspace{-1ex}
\begin{align*}
    \begin{cases}
        A(\bm{u})\partial_t\bm{u}-\mathrm{div}\bigl(\alpha(\bm{u}) B^*(\bm{u}) \partial \gamma({B}(\bm{u})\nabla\bm{u})+\kappa \nabla \bm{u}\bigr)+\nabla_{\bm{u}} G(x, \bm{u})
        \\
        \qquad+[\nabla\alpha](\bm{u})\gamma({B}(\bm{u})\nabla\bm{u})+\alpha(\bm{u})\partial \gamma({B}(\bm{u})\nabla\bm{u}):[\nabla{B}](\bm{u})\nabla\bm{u}\ni \bm{0} \mbox{ in $ Q $,}
        \\[0.5ex]
        \bigl(\alpha(\bm{u})B^*(\bm{u}) \partial \gamma({B}(\bm{u})\nabla\bm{u})+\kappa \nabla \bm{u} \bigr) \bm{n}_\Gamma\ni \bm{0} \mbox{ on } \Sigma,
        \\[0.5ex]
        \bm{u}(0,x)=\bm{u}_0(x), ~x\in\Omega.
    \end{cases}
\end{align*}\vspace{-1.75ex}

This system is a gradient flow of the following energy functional:\vspace{-1.5ex}
\begin{gather}
    \mathscr{E}: \bm{u} = {}^\top \bigl[ u_1, \dots, u_M \bigr] \in [H^{1}(\Omega)]^M \mapsto \mathscr{E}(\bm{u}) := \frac{\kappa}{2}\int_\Omega|\nabla \bm{u}|^2\,dx
  \nonumber
\\
  +\int_\Omega \alpha(\bm{u})\gamma(B(\bm{u})\nabla \bm{u}) \, dx+\int_\Omega G(x,\bm{u}) \, dx \in [0, \infty).
  \label{energy01}
\end{gather}\vspace{-2.75ex}

The functional $\mathscr{E}=\mathscr{E}(\bm{u})$ generalizes the free energies proposed in \cite{berkels2006cartoon,MR1752970}. In particular, system (S) includes as applications the mathematical models proposed in the following two research areas:
\begin{itemize}
    \item Anisotropic image processing models with automatic orientation adaptation mechanisms (cf. \cite{berkels2006cartoon}), where $\gamma$ is a real convex function characterizing image anisotropy, while $B(\bm{u})$ is an unknown-dependent operator (rotation) responsible for orientation adjustment.
    \item Mathematical models of grain boundary motion (cf. \cite{MR1752970}), where $A(\bm{u})$ and $\alpha(\bm{u})$ represent the unknown-dependent mobility of grain boundaries.
\end{itemize}

\noindent
Thus, analyzing the gradient system (S) establishes a unified framework linking image processing and materials science. Moreover, in both fields, incorporating the gradient $\nabla_{\bm{u}} G(x,\bm{u})$ of a smooth potential $G(x,\bm{u})$ allows for more realistic phenomena.

Several attempts have been made to construct a unified framework for the pseudo-parabolic gradient system associated with $\mathscr{E}$, obtained by adding the energy-dissipation term $-\Delta \partial_t \bm{u}$ to (S) (cf.~\cite{arAMSU2025}). However, models for image processing and grain-boundary motion are originally formulated as parabolic gradient flows (cf.~\cite{MR1752970,Shirakawa1818eng}). Thus, a unified analytical framework for these parabolic models remains open.

In this paper, we resolve this open problem by analyzing system (S), thereby completing the unification of the framework. Accordingly, we introduce a notion of solution that accounts for energy dissipation, which we refer to as an energy-dissipative solution. The main result concerns the existence of such solutions and is presented in Theorem \ref{mainThm1}.

The results of this paper establish a unified analytical framework for parabolic models. This framework bridges image processing and materials science and provides a solid theoretical foundation for the study of advanced problems arising in both fields.

\section{Preliminaries}\label{secpre}\vspace{-1ex}
We begin by specifying the notation adopted throughout the paper.

\noindent
\underline{\it\textbf{Real analysis.}} We define $a \wedge b : = \min\{ a , b \} , \mbox{ for all } a , b \in [-\infty, \infty]$. Let $d \in \N$ be fixed dimension. We denote by $ | x | $ and $ x \cdot y $ the Euclidean norm of $ x \in \R^d$ and the standard scalar product of $ x , y \in \R^d$, respectively, i.e.,
$
  | x | : = \sqrt{x_1^2 + \cdots + x_d^2} \quad \mbox{and} \quad x \cdot y := x_1 y_1 + \cdots + x_d y_d,\mbox{ for all } x , \, y\in \R^d.
$
Additionally, we note the following elementary fact:\vspace{-1ex}
\begin{description}
  \item[\textbf{(Fact 1)}]Let $ m \in \N $ be fixed. If $ \{A_k\}_{k=1}^m \subset \R $ and $ \{ a_n^k \}_{n\in\N}\subset\R $, 
  $ k = 1, \dots , m $ satisfies that:\vspace{-1ex}
  \[ \ulim_{n \to \infty} a^k_n \geq A_k, \mbox{ for } k = 1, \dots , m, \mbox{ and }\olim_{n \to \infty} \sum_{k=1}^{m} a_n^k \leq \sum_{k=1}^{m} A_k.\vspace{-1.5ex} \]
  Then, $ \lim_{n \rightarrow \infty } a_n^k = A_k $, for $ k = 1 , \dots , m $. \vspace{-1ex}
\end{description}
\underline{\it\textbf{Abstract functional analysis.}} For an abstract Banach space $ X $, we denote by $| \cdot |_X$ the norm of $ X $, and denote by $ \langle \cdot , \cdot \rangle_X $ the duality pairing between $ X $ and its dual $ X^* $. In particular, when $ X $ is a Hilbert space, we denote by $( \cdot , \cdot )_X$ its inner product. For Banach spaces $ X_1 , \dots ,X_d $ with $ 1 < d \in \N$, let $ X_1 \times \dots \times X_d $ be the product Banach space which has the norm \vspace{-1.5ex}
\[ | \cdot |_{ X_1\times \dots \times X_d } : = \left(| \cdot |_{X_1}^2 + \dots + | \cdot |_{X_d}^2\right)^\frac{1}{2} .\vspace{-1ex}\] 
In addition, in the case where the domain is the whole space we shall denote the norm in the following:\vspace{-1ex}
\[
  \|\cdot\|_{W^{p.q}}:=|\cdot|_{W^{p.q}(\R^d;\R^k)}, \mbox{ for all $ d,k\in\N$, $ p \in \N\cup \{0\}, $ and $q \in[1,\infty]$.}\vspace{-1ex}
\]
Furthermore, for normed spaces $E$ and $F$, we denote by $\mathcal{L}(E,F)$ the space of all continuous linear maps from $E$ into $F$. 
\vspace{0.5ex}

\noindent
\underline{\it\textbf{Convex analysis.}} For any proper lower semi-continuous (l.s.c.) and convex function $ \Psi : X \rightarrow (-\infty,\infty]$ on a Hilbert space $ X $, we denote by $ D( \Psi ) $ the effective domain of $ \Psi $, and denote by $ \partial \Psi $ the subdifferential of $ \Psi $. More precisely, for each $ w \in X $, the value $ \partial \Psi( w ) $ is defined as the set of all elements $ w^* \in X $ that satisfy the variational inequality\vspace{-0.75ex}
\[( w^*, x - w )_X \leq \Psi (x) - \Psi (w), \mbox{ for any } x \in D ( \Psi ),\vspace{-0.75ex} \]
and the set $ D ( \partial \Psi ) := \{ x \in X \,|\, \partial \Psi (x) \neq \emptyset \}$ is called the domain of $ \partial \Psi $. We often use the notation $`` [w, w^*] \in \partial \Psi $ in $ X \times X "$ to mean that $`` w^* \in \partial \Psi ( w ) $ in $ X $ for $ w \in D ( \partial \Psi ) "$, by identifying the operator $ \partial \Psi $ with its graph in $ X \times X $.\vspace{-1ex}
\begin{example}\label{ex1}
  Let $\gamma:\R^{M\times N}\to[0,+\infty)$ be a convex function in $C^{0,1}(\R^{M\times N})$. For $\varepsilon\ge0$, define ${\gamma_\varepsilon}$ by:\vspace{-1ex}
  \begin{equation*}
    \gamma_\varepsilon :=\left\{
      \begin{aligned}
        &\gamma, &&\mbox{ if }\varepsilon=0,
        \\
        &\rho_\varepsilon * \gamma, &&\mbox{ otherwise},
      \end{aligned}\right. \mbox{ on }\R^{M\times N},\vspace{-1ex}
  \end{equation*}
  where $\rho_\varepsilon$ is the standard mollifier. Then, the following two items hold.
  \begin{description}
    \item[(I)] Let $\{\Phi_\varepsilon\}_{\varepsilon\geq0}$ be a sequence of functionals on $[L^2(\Omega)]^{M\times N}$, defined as:\vspace{-1ex}
  \[
    \Phi_\varepsilon:W\in[L^2(\Omega)]^{M\times N}\mapsto\Phi_\varepsilon(W):=\int_\Omega\gamma_\varepsilon(W)\,dx\in[0,\infty).\vspace{-1ex}
  \]
  Then, for every $\varepsilon\in[0,\infty)$, $\Phi_\varepsilon$ is the proper, l.s.c., and convex function, 
  $
  D(\Phi_\varepsilon)=D(\partial\Phi_\varepsilon)=[L^2(\Omega)]^{M\times N}, \mbox{ and } \mbox{ for any }W\in[L^2(\Omega)]^{M\times N},
  $ 
  \begin{equation*}
    \partial\Phi_\varepsilon(W):=
    \left\{
      \begin{aligned}
        &\{\nabla\gamma_\varepsilon(W)\}, \mbox{ if }\varepsilon>0,
        \\
        &\{W^*\in[L^2(\Omega)]^{M\times N}~|~W^*\in\partial\gamma(W)\mbox{ a.e. in } \Omega,\}, \mbox{ if }\varepsilon=0,
      \end{aligned}
    \right.
  \end{equation*}\vspace{-3ex}
    \item[(II)] Let any open interval $I\subset(0,T)$, and let $\{\widehat{\Phi}_\varepsilon^I\}_{\varepsilon\geq0}$ be a sequence of functionals on $L^2(I;[L^2(\Omega)]^{M\times N})$, defined as:
  \[
  \widehat{\Phi}_\varepsilon^I:W\in L^2(I;[L^2(\Omega)]^{M\times N})\mapsto\widehat{\Phi}_\varepsilon^I(W):=\int_I\Phi_\varepsilon(W(t))\,dt\in[0,\infty).\vspace{-1ex}
  \]
  Then, any $\varepsilon\in[0,\infty)$, $\widehat{\Phi}^I_\varepsilon$ is the proper, l.s.c., and convex function, 
  $
  D(\widehat{\Phi}_\varepsilon^I)=D(\partial\widehat{\Phi}_\varepsilon^I)=L^2(I;[L^2(\Omega)]^{M\times N}),\mbox{ and }\mbox{ for any }W\in D(\widehat{\Phi}_\varepsilon^I),
  $\vspace{-1ex}
  \begin{align*}
    &\partial\widehat{\Phi}_\varepsilon^I(W)
    \\
    &=\{\tilde{W}^*\in D(\widehat{\Phi}_\varepsilon^I)\,|\,\tilde{W}^*(t)\in \partial \Phi_\varepsilon(W(t))\mbox{ in }[L^2(\Omega)]^{M\times N}, \mbox{ a.e. }t\in I\}
    \\
    &=
    \left\{
      \begin{aligned}
        &\{\nabla\gamma_\varepsilon(W)\}, \mbox{ if }\varepsilon>0,
        \\
        &\{W^*\in D(\widehat{\Phi}_\varepsilon^I)\,|\,W^*\in\partial\gamma(W)\mbox{ a.e. in } I\times \Omega\}, \mbox{ if }\varepsilon=0,
      \end{aligned}
    \right.
  \end{align*}
  \end{description}
\end{example}\vspace{-2.5ex}

\begin{definition}[{Mosco-convergence}: cf.\cite{MR0298508}]\label{dfnmosco}
    Let $ X $ be a Hilbert space. Let $ \Psi : X \rightarrow ( -\infty , \infty ] $ be a proper, l.s.c., and convex function, and let $ \{ \Psi_n \}_{ n \in \N } $ be a sequence of proper, l.s.c., and convex functions $ \Psi_n : X \rightarrow ( -\infty , \infty ] $, $ n \in \N $. Then, we say that $ \Psi_n \to  \Psi $ on $ X $ in the sense of Mosco, iff. the following two conditions are fulfilled:\vspace{-1ex}
  \begin{description}
    \item[(M1) (Optimality)] For any $w_0 \in D ( \Psi )$, there exists a sequence $ \{w_n\}_{ n \in \N }$ $\subset X $ such that $ w_n \rightarrow w_0 $ in $ X $ and $ \Psi_n ( w_n ) \rightarrow \Psi ( w_0 ) $ as $ n \rightarrow \infty $, \vspace{-1ex}
    \item[(M2) (Lower-bound)] $\ulim_{n \to \infty} \Psi_n ( w_n ) \geq \Psi ( w_0 )$ if $w_0 \in X, \{ w_n \}_{ n \in \N } \subset X $, and $w_n \rightarrow w_0 $ weakly in $ X $ as $ n \rightarrow \infty $.
  \end{description}
\end{definition}\vspace{-2ex}
\begin{definition}[{$ \Gamma $-convergence}; cf.\cite{MR1201152}]\label{dfngamma}
  Let $ X $ be a Hilbert space. Let $ \Psi : X \rightarrow ( -\infty , \infty ] $ be a proper and l.s.c. function, and let $ \{ \Psi_n \}_{ n \in \N } $ be a sequence of proper and l.s.c. functions $ \Psi_n : X \rightarrow ( -\infty , \infty ] $, $ n \in \N $. Then, we say that {$ \Psi_n \to \Psi $} on $ X $ in the sense of $ \Gamma $-convergence, iff. the following two conditions are fulfilled:\vspace{-1ex}
  \begin{description}
      \item[{($\mathbf{\Gamma}$1) (Optimality)}] For any $w_0 \in D ( \Psi )$, there exists a sequence $ \{w_n\}_{ n \in \N }$ $\subset X $ such that $ w_n \rightarrow w_0 $ in $ X $ and $ \Psi_n ( w_n ) \rightarrow \Psi ( w_0 ) $ as $ n \rightarrow \infty $, \vspace{-1ex}
      \item[{($\mathbf{\Gamma}$2) (Lower-bound)}] $\ulim_{n \to \infty} \Psi_n ( w_n ) \geq \Psi ( w_0 )$ if $w_0 \in X, \{ w_n \}_{ n \in \N } \subset X $, and $w_n \rightarrow w_0 $ in $ X $ as $ n \rightarrow \infty $.
  \end{description}
\end{definition}
\vspace{-2.5ex}
\begin{remark}\label{rem2}
  Under convexity, Mosco convergence implies $\Gamma$-convergence; thus, $\Gamma$-convergence can be viewed as a weaker notion of Mosco convergence. Moreover, for convex functions, the following holds:\vspace{-1ex}
\begin{description}
  \item[(Fact 2)] (cf.\cite[Theorem 3.66]{MR0773850} and \cite[Chapter 2]{Kenmochi81}) Let $X$ be a Hilbert space. Let $\Psi, \Psi_n : X \to (-\infty,\infty]$, $n\in\mathbb{N}$, be proper, l.s.c., and convex functions such that $\Psi_n \to \Psi$ on $X$ in the sense of $\Gamma$-convergence as $n \to \infty$. Let us assume that 
  \begin{equation*}
    \left\{
    \begin{aligned}
      &[z, z^*] \in X \times X ,~ [z_n, z_n^*] \in \partial \Psi_n \mbox{ in } X \times X,~n \in \N,
      \\
      &z_n \rightarrow z^* \mbox{ in } X \mbox{ and } z_n^* \rightarrow z^* \mbox{ weakly in } X \mbox{ as } n \rightarrow \infty.  
    \end{aligned}
    \right.\vspace{-1ex}
  \end{equation*}
  Then, it holds that:\vspace{-1ex}
  \[ [ z , z^* ] \in \partial \Psi \mbox{ in } X \times X , \mbox{ and } \Psi_n ( z_n ) \rightarrow \Psi ( z ) \mbox{ as } n \rightarrow \infty.\vspace{-1.5ex} \]
  \item[(Fact 3)](cf.\cite[Lemma 4.1]{MR3661429} and \cite[Appendix]{MR2096945}) Let $ X $ be a Hilbert space, $ d \in \N $ be dimension constant, and $ A \subset \R^d $ be a bounded open set. Let $ \Psi : X \rightarrow ( -\infty , \infty ] $ and $ \Psi_n : X \rightarrow ( -\infty , \infty ] $, $ n \in \N $, be proper, l.s.c., and convex functions on $ X $ such that $ \Psi_n \rightarrow \Psi $ on $ X $, in the sense of $ \Gamma $-convergence, as $ n \rightarrow \infty $. Then, a sequence $ \{ \widehat{ \Psi }_n^A \}_{ n \in \N } $ of proper, l.s.c., and convex functions on $ L^2 ( A ; X ) $, defined as:\vspace{-1ex}
  \begin{equation*}
    z \in L^2 ( A ; X ) \mapsto \widehat{ \Psi }^A_n ( z ) : = \left\{
      \begin{aligned}
        & \int_A \Psi _n ( z ( t ) ) \,dt, \mbox{ if } \Psi_n ( z ) \in L^1 ( A ), 
        \\
        & \infty, \mbox{ otherwise}, 
      \end{aligned}
    \right.\vspace{-1.5ex}
  \end{equation*}
  $\mbox{ for }n \in \N$, converges to a proper, l.s.c., and convex function $ \widehat{ \Psi }^A $ on $ L^2 ( A ; X ) $, defined as:\vspace{-1.5ex}
  \begin{equation*}
    z \in L^2 ( A ; X ) \mapsto \widehat{ \Psi }^A ( z ) : = \left\{
      \begin{aligned}
        & \int_A \Psi ( z ( t ) ) \,dt, \mbox{ if } \Psi ( z ) \in L^1 ( A ), 
        \\
        & \infty, \mbox{ otherwise}, 
      \end{aligned}
    \right.\vspace{-1.5ex}
  \end{equation*}
  on $ L^2 ( A ; X ) $, in the sense of $ \Gamma $-convergence, as $ n \rightarrow \infty $.
\end{description}
\end{remark}
\begin{example}\label{ex2}
  Let $\varepsilon_0\geq0$ be arbitrary fixed constant, and let $\gamma$ and $\{\gamma_\varepsilon\}_{\varepsilon\geq0}$ be as in Example \ref{ex1}, respectively. Then, the following three items hold.\vspace{-1ex}
  \begin{description}
    \item[(I)] $\gamma_\varepsilon\rightarrow\gamma_{\varepsilon_0}$ on $\R^{M\times N}$, in the sense of Mosco, as $\varepsilon\rightarrow\varepsilon_0$.\vspace{-1ex}
    \item[(II)] Let $\{\Phi_\varepsilon\}_{\varepsilon\geq0}$ be the sequence of proper, l.s.c., and convex functions on $[L^2(\Omega)]^{M\times N}$, as in Example \ref{ex1}(I). Then,\vspace{-1ex}
    \[
    \Phi_\varepsilon\rightarrow\Phi_{\varepsilon_0} \mbox{ on }[L^2(\Omega)]^{M\times N} \mbox{ in the sense of Mosco, as }\varepsilon\rightarrow\varepsilon_0.\vspace{-2ex}
    \] 
    \item[(III)] Let $I\subset(0,T)$ be an open interval, and let $\{\widehat{\Phi}_{\varepsilon}^I\}_{\varepsilon\geq0}$ be the sequence of proper, l.s.c., and convex functions on $L^2(I;[L^2(\Omega)]^{M\times N})$, as in Example \ref{ex1}(II). Then, \vspace{-1ex}
    \[
    \widehat{\Phi}^I_\varepsilon\rightarrow\widehat{\Phi}^I_{\varepsilon_0} \mbox{ on }L^2(I;[L^2(\Omega)]^{M\times N}), \mbox{ in the sense of Mosco, as }\varepsilon\rightarrow\varepsilon_0.\vspace{-1ex}
  \]
  \end{description}
\end{example}
\noindent
\underline{\it\textbf{Notations for the time-discretization}}\label{deftseq}
  Let $\tau>0$ be a time step size and define $t_i=i\tau$ for $i=0,1,2,\dots$. Let $X$ be a Banach space. Given ${(t_i,u_i)}{i=0}^\infty \subset [0,\infty)\times X$, we define the piecewise interpolations $[\overline{u}]\tau$, $[\underline{u}]\tau \in L^\infty_{\mathrm{loc}}([0,\infty);X)$ and $[u]\tau \in W^{1,2}_{\mathrm{loc}}([0,\infty);X)$ as follows:\vspace{-0.5ex}
\begin{equation*}
  \left\{
  \begin{aligned}
    &[\overline{u}]_\tau(t):=\chi_{(-\infty,0]}u_{0}+\sum_{i=1}^{\infty}\chi_{(t_{i-1},t_i]}(t)u_{i},
    \\[-0.5ex]
    &[\underline{u}]_\tau(t):=\sum_{i=0}^{\infty}\chi_{(t_i,t_{i+1}]}(t)u_{i},\hspace*{25ex}\mbox{in }X,\mbox{ for any } t\geq0,
    \\[-0.5ex]
    &[u]_\tau(t):=\sum_{i=1}^{\infty}\chi_{(t_{i-1},t_i]}(t)\biggl(\frac{t-t_{i-1}}{\tau}u_{i}+
    \frac{t_i-t}{\tau}u_{i-1}\biggr),
  \end{aligned}
  \right.
\end{equation*}
where $\chi_E:\mathbb{R}\to\{0,1\}$ denotes the characteristic function of a set $E\subset\mathbb{R}$.
\vspace{0.5ex}

\noindent
\underline{\it\textbf{Notations of basic differential operators.}}\label{tensordef}
  Let $A = [a_{ij}]$ and $B = [b_{ij}]$ be matrices in $\mathbb{R}^{M \times N}$, 
with entries $a_{ij}, b_{ij} \in \mathbb{R}$ for $i = 1,\ldots,M$ and $j = 1,\ldots,N$. 
We define the Frobenius inner product and norm by
$
A : B = \sum_{i=1}^M \sum_{j=1}^N a_{ij} b_{ij}\mbox{ and }\|A\| = \sqrt{A : A}, \mbox{ respectively.}
$

For $M > 1$, and a vector-valued function $\bm{z} = [z_i] \in [L^1_{\mathrm{loc}}(\Omega)]^M$. 
Its distributional gradient is given by
$ \nabla \bm{z} = {}^\top(\nabla z_1, \ldots, \nabla z_M) 
\in \mathcal{D}'(\Omega)^{M \times N} $. 
We also denoted by $\Delta \bm{z} := [\Delta z_i] \in [\mathscr{D}'(\Omega)]^M$ as the usual Laplace operator, in the distributional sense. In particular, let $\Delta_N$ denote the Laplacian with homogeneous Neumann boundary condition, defined as 
$
D(\Delta_N) := 
\left\{ \tilde{\bm{z}} \in [H^2(\Omega)]^M \;\middle|\; 
\nabla \tilde{\bm{z}}|_\Gamma \cdot n_\Gamma = 0 
\text{ in } [H^{1/2}(\Gamma)]^m \right\}
$, and for $\bm{z} \in D(\Delta_N)$ we set
$
\Delta_N \bm{z} := [\Delta_N z_i] = \Delta \bm{z} \in [L^2(\Omega)]^M.
$ The operator $-\Delta_N$ is identified as a linear isomorphism via the Green-type formula (cf. \cite[Propositon 5.6.2]{MR2192832}):\vspace{-1ex}
\[
- \int_\Omega \Delta_N \bm{z} \cdot \bm{w} \, dx 
= \int_\Omega \nabla \bm{z} : \nabla \bm{w} \, dx, 
 \mbox{ for any} (\bm{z},\bm{w}) \in D(\Delta_N) \times [L^2(\Omega)]^M.\vspace{-1ex}
\]

Finally, for a matrix-valued function $\bm{Z} = [z_{ik}] \in [L^2(\Omega)]^{M \times N}$, 
we define its distributional divergence by $
\diver \bm{Z} := 
\sum_{k=1}^N \partial_k z_{ik}
\in \mathcal{D}'(\Omega)^M.
$\vspace{-2ex}
\section{Main results}\label{sec:main}\vspace{-1ex}
The Main Theorems are discussed under the following assumptions.
\begin{description}
  \item[(A0)]$ \kappa > 0 $ is a fixed constant. Also, we let $H:=L^2(\Omega)$, $V:=H^1(\Omega)$, $\mathscr{H}:=L^2(0,T;H)$, $\mathscr{V}:=L^2(0,T;V)$.\vspace{-2ex}
  \item[(A1)] $\alpha : \R^M \rightarrow [0,\infty)$ is belongs to $\alpha \in W^{2,\infty}(\R^M) \cap C^2(\R^M)$.\vspace{-1.5ex}
  \item[(A2)] $G : (x, \bm{u}) \in \Omega \times \R^M \rightarrow G(x, \bm{u}) \in [0,\infty)$ is a fixed function, such that $ G(\cdot, \bm{u}) \in L^1(\Omega) $ for any $ \bm{u} \in \R^M $, $ G(x, \cdot) \in C^{1, 1}(\R^M) $ for a.e. $ x \in \Omega $, and furthermore, the gradient $ \nabla_{\bm{u}}G = {}^\top\bigl[ \partial_{u_i} G \bigr]_{1 \leq i \leq M} $ and Hessian $ \nabla_{\bm{u}}^2 G = \bigl[ \partial_{u_i} \partial_{u_j} G \bigr]_{1 \leq i, j \leq M} $ for variable $ \bm{u} = {}^\top[u_1, \dots, u_M] \in \R^M $ are bounded over $ \Omega \times \R^M $.\vspace{-2ex}
  \item[(A3)] $B: \bm{u} \in \R^M \mapsto B(\bm{u}) W := B_0(\bm{u})W B_1(\bm{u}) \in \R^{M \times N}$ is a linear operator defined by using $ B_0 \in [W^{2,\infty}(\R^M) \cap C^2(\R^M)]^{M \times M}$ and $B_1 \in [W^{2,\infty}(\R^M) \cap C^2(\R^M)]^{N \times N}$. Also, $[\nabla B]:\bm{u}\in\R^M \mapsto [\nabla B](\bm{u})W := [\nabla B_0](\bm{u})W B_1(\bm{u}) + B_0(\bm{u})W [\nabla B_1](\bm{u}) \in (\R^{M \times N})^M$ is the differential of $ B $, for $W,Z\in \R^{M\times N} $,
\vspace{-1ex}
\begin{align*}
  &[\nabla B_0](\bm{v})WB_1(\bm{v}):=\left[[\partial_{v_i}B_0](\bm{v})WB_1(\bm{v})\right]_{1\leq i\leq M}\in (\R^{M\times N})^M,
  \\
  &B_0(\bm{v})W[\nabla B_1](\bm{v}):=\left[B_0(\bm{v})W[\partial_{v_i}B_1](\bm{v})\right]_{1\leq i\leq M}\in (\R^{M\times N})^M,
  \\
  &Z:[\nabla B_0](\bm{v})WB_1(\bm{v}):=\left[Z:[\partial_{v_i}B_0](\bm{v})WB_1(\bm{v})\right]_{1\leq i\leq M}\in\R^M,
  \\
  &Z:B_0(\bm{v})W[\nabla B_1](\bm{v}):=\left[Z:B_0(\bm{v})W[\partial_{v_i}B_1](\bm{v})\right]_{1\leq i\leq M}\in\R^M.
\end{align*}\vspace{-4.5ex}

        Moreover, $ B^*: \bm{u} \in \R^M \mapsto B^*(\bm{u}) W := \,{}^\top B_0(\bm{u}) W \,{}^\top B_1(\bm{u}) \in \R^{M \times N}$.\vspace{-1ex}
  \item[(A4)] $A\in W^{1,\infty}(\R^M;\R^{M \times M})$ is a given function, which has a positive constant $ C_A > 0 $ satisfying $^\top \bm{w}A(\bm{v})\bm{w}\geq C_A|\bm{w}|^2, \mbox{ for all }\bm{v},\bm{w}\in\R^M. $\vspace{-1ex}
  \item[(A5)] $0 \leq \gamma \in C^{0, 1}(\R^{M \times N})$ is a convex function satisfying that there exists a constant $C_\gamma>0$ such that $\gamma(W)\leq C_\gamma(|W|+1),~\mbox{ for all } W\in\R^{M\times N}.$\vspace{-1.5ex}
  \item[(A6)] $\bm{u}_0$ is a fixed initial data such that $\bm{u}_0\in [V]^M$.
\end{description}
\begin{remark}\label{Rem.(A2)}
  Note that assumption (A2) yields the following condition.
    \begin{description}
      \item[(A2)$'$]
        The function $ x \in \Omega \mapsto  G(x, \bm{0}) $ belongs to $ L^1(\Omega) $, and there exists a constant $ L_G $, independent of variables $ x \in \Omega $ and $ \bm{u} \in \R^M $, such that
          \begin{gather*}
            \bigl|G(x, \bm{u}) -G(x, \bm{\tilde{u}})\bigr| +\bigl|\nabla_{\bm{u}}G(x, \bm{u}) -\nabla_{\bm{u}}G(x, \bm{\tilde{u}})\bigr| \leq L_G |\bm{u} -\bm{\tilde{u}}|, \\
            \mbox{for a.e. $ x\in \Omega $, and all $ \bm{u}, \bm{\tilde{u}} \in \R^M $.}
          \end{gather*}
    \end{description}
\end{remark}\vspace{-1ex}
Next, let us give the definition of the solution to the system (S), which is called an energy dissipative solution.
\begin{definition}
  A function $\bm{u} \in [\mathscr{H}]^M$ is called a energy dissipative solution to the system (S) if and only if $\bm{u}$ fulfills the following conditions:\vspace{-1ex}
  \begin{description}
      \item[(S0)] $\bm{u}\in W^{1,2}(0,T;[H]^M)\cap L^\infty(0,T; [V]^M)$, $ \bm{u}(0) = \bm{u}_0 $ in $ [H]^M $.\vspace{-1ex}
    \item[(S1)] There exists a function $\bm{w}^*\in L^2(0,T;[H]^{M\times N})$ such that:\vspace{-1ex}
      \begin{align*}
        \bm{w}^*(t,x)\in\partial\gamma(B(\bm{u}(t,x))\nabla \bm{u}(t,x)) \mbox{ in } \R^{M\times N}, \mbox{ for a.e. }(t,x) \in Q,
      \end{align*}
      and $\bm{u}$ solves the following variational inequality: 
    \begin{gather*}
    (A(\bm{u}(t))\partial_t\bm{u}(t),\bm{u}(t)-\bm{\varphi})_{[H]^M}+\kappa(\nabla\bm{u}(t),\nabla(\bm{u}(t)-\bm{\varphi}))_{[H]^{M\times N}}
    \\
    +(\nabla_{\bm{u}} G(x,\bm{u}(t))+[\nabla\alpha](\bm{u}(t))\gamma({B}(\bm{u}(t))\nabla\bm{u}(t)),\bm{u}(t)-\bm{\varphi})_{[H]^M}
    \\
    +(\alpha(\bm{u}(t))\bm{w}^*(t):[\nabla{B}](\bm{u}(t))\nabla\bm{u}(t),\bm{u}(t)-\bm{\varphi})_{[H]^M}
    \\
    +\int_{\Omega}\alpha(\bm{u}(t))\gamma(B(\bm{u}(t))\nabla \bm{u}(t))\,dx\leq\int_{\Omega}\alpha(\bm{u}(t))\gamma (B(\bm{u}(t))\nabla\bm{\varphi})\,dx,
    \\
    \mbox{ for all } \bm{\varphi}\in [V]^M, \mbox{ and a.e. } t\in(0,T).
    \end{gather*}\vspace{-3.75ex}
    \item[(S2)] $\bm{u}$ satisfies the following energy inequality: \vspace{-1ex}
    \begin{gather}
    \frac{C_A}{4}\int_{s}^{t}|\partial_t\bm{u}(\sigma)|^2_{[H]^M}\,d\sigma
      +E_0(\bm{u}(t))\leq E_0(\bm{u}(s)),
    \label{ene-inq1}
  \end{gather}
  for a.e. $ s \in [0, T) $ including $s=0$, and any $ t \in [s, T] $.
  \end{description}
\end{definition}\vspace{-1ex}
\begin{theorem}\label{mainThm1}
  Under the assumptions (A0)--(A7), the system (S) admits a energy dissipative solution $\bm{u}$.
\end{theorem}
\section{Proof of Theorem 1}\vspace{-1.5ex}
For the proof of Theorem 1, we use a time-discretization method. Let $\nu,\varepsilon,\mu\in(0,1)$ be fixed constants. Let $m\in\N$ denote the number of subdivisions of the time interval $(0,T)$. We set $\tau := T/m$, and define the time sequence $\{t^i\}_{i=1}^m$ by $t^i := i\tau$ for $i=1,\dots,m$.

Based on these, we consider the time-discretization scheme (AP)$^\tau_{\nu,\varepsilon,\mu}$ as an approximating problem of (S):

\noindent
(AP)$^\tau_{\nu,\varepsilon,\mu}$: \, For any $\bm{u}^{0}_{\nu,\varepsilon,\mu}\in [W^{1,p}(\Omega)]^M$, find a $ \{\bm{u}^i_{\nu,\varepsilon,\mu}\}_{i=1}^m \subset [W^{1,p}(\Omega)]^M $ such that
\begin{gather}
  \frac{1}{\tau}(A(\bm{u}^{i-1}_{\nu,\varepsilon,\mu})(\bm{u}^{i}_{\nu,\varepsilon,\mu}-\bm{u}^{i-1}_{\nu,\varepsilon,\mu}),\bm{\varphi})_{[H]^M}+\frac{\mu}{\tau}(\nabla(\bm{u}^{i}_{\nu,\varepsilon,\mu}-\bm{u}^{i-1}_{\nu,\varepsilon,\mu}),\nabla\bm{\varphi})_{[H]^{M\times N}}\nonumber
  \\\nonumber
  +(\alpha(\bm{u}^{i}_{\nu,\varepsilon,\mu})  \nabla\gamma_\varepsilon({B}(\bm{u}^{i}_{\nu,\varepsilon,\mu})\nabla\bm{u}^{i}_{\nu,\varepsilon,\mu}),B(\bm{u}^{i}_{\nu,\varepsilon,\mu})\nabla\bm{\varphi})_{[H]^{M\times N}}
  \\\nonumber
  +\nu \int_\Omega\nabla\Upsilon_p(\nabla \bm{u}^{i}_{\nu,\varepsilon,\mu}):\nabla\bm{\varphi}\,dx+\kappa(\nabla\bm{u}^{i}_{\nu,\varepsilon,\mu},\nabla\bm{\varphi})_{[H]^{M\times N}}
  \\\nonumber
  +(\nabla_{\bm{u}} G(x,\bm{u}^{i}_{\nu,\varepsilon,\mu}),\bm{\varphi})_{[H]^M}+([\nabla\alpha](\bm{u}^{i-1}_{\nu,\varepsilon,\mu})\gamma_\varepsilon({B}(\bm{u}^{i-1}_{\nu,\varepsilon,\mu})\nabla\bm{u}^{i-1}_{\nu,\varepsilon,\mu}),\bm{\varphi})_{[H]^M}
  \\\nonumber
  +(\alpha(\bm{u}^{i-1}_{\nu,\varepsilon,\mu})\nabla\gamma_\varepsilon({B}(\bm{u}^{i-1}_{\nu,\varepsilon,\mu})\nabla\bm{u}^{i-1}_{\nu,\varepsilon,\mu}):[\nabla{B}](\bm{u}^{i-1}_{\nu,\varepsilon,\mu})\nabla\bm{u}^{i-1}_{\nu,\varepsilon,\mu},\bm{\varphi})_{[H]^M}=0,
  \\\
  \label{3TimeDis-02}
  \mbox{ for all } \bm{\varphi}\in [W^{1,p}(\Omega)]^M, \mbox{ and }i=1,\dots,m.
\end{gather}
In this context, $p \in (2,\infty) \cap [N,\infty)$ is a fixed constant, and $ 0 \leq \Upsilon_p \in C^1(\R^{M \times N}) $ is a given function, which has a constant $ C_\Upsilon > 0 $ satisfying for any $W,W_1,W_2\in\R^{M\times N}$,\vspace{-1ex}
  \begin{gather}
    \frac{1}{C_\Upsilon}(|W|^p-1)\leq \Upsilon_p(W)\leq C_\Upsilon(|W|^p+1), \mbox{ and }\nonumber
  \\
    (\nabla \Upsilon_p(W_1)-\nabla \Upsilon_p(W_2)):(W_1-W_2)\geq C_\Upsilon|W_1-W_2|^p.\label{upsip}
  \end{gather}
Moreover, ${\gamma_\varepsilon}$ denotes a regularization sequence of $\gamma$ defined by $\gamma_\varepsilon := \rho_\varepsilon * \gamma$ for $\varepsilon>0$, using the standard mollifier $\rho_\varepsilon$. Additionally, (AP)$_{\nu,\varepsilon,\mu}^{\tau}$ can be regarded as a discrete gradient descent process for an approximating energy $\mathscr{E}_{\nu,\varepsilon}$, defined by \vspace{-1ex}
\begin{gather*}
  \mathscr{E}_{\nu,\varepsilon}:\bm{u}\in[W^{1,p}(\Omega)]^M\mapsto \mathscr{E}_{\nu,\varepsilon}:=\frac{\kappa}{2}\int_\Omega|\nabla \bm{u}|^2\,dx+\nu\int_\Omega \Upsilon_p(\nabla \bm{u})\,dx
\\
  +\int_\Omega \alpha(\bm{u})\gamma_\varepsilon(B(\bm{u})\nabla \bm{u}) \, dx+\int_\Omega G(x,\bm{u}) \, dx.
\end{gather*}
As is easily seen, \vspace{-1ex}
    \begin{gather}
            \gamma_\varepsilon \to \gamma \mbox{ uniformly on $ \R^{M\times N} $ as $ \varepsilon \downarrow 0 $, and }\label{gamma_ep}
            \\
            \|\nabla \gamma_\varepsilon\|_{L^\infty} \leq \|\nabla \gamma\|_{L^\infty}, \mbox{ for $ \varepsilon \in (0, 1) $.}\nonumber
    \end{gather}
    The uniform convergence of $ \{ \gamma_\varepsilon \}_{\varepsilon \in (0, 1)} $ implies 
\begin{gather}\label{gammaconv}
    \mathscr{E}_{\nu,\varepsilon}\rightarrow \mathscr{E}_{\nu,0}\mbox{ on }[H]^M,\mbox{ in the sense of $\Gamma$-convergence, as }\varepsilon\downarrow0.
\end{gather}
The following Lemmas provides a fundamental result for the scheme (AP)$_{\nu,\varepsilon,\mu}^{\tau}$, which serves as a basis for the Theorem 1 of this paper. These lemmas can be proved in the same way as in \cite{arAMSU2025}. 
    \begin{lemma}\label{003Thm1}
        There exists a sufficiently small constant $\tau_0:=\tau_0(\nu,\varepsilon,\mu)\in (0,1)$ such that for any $ \tau \in (0,\tau_0) $, $\mathrm{(AP)}^\tau_{\nu,\varepsilon,\mu}$ admits a unique solution $\{\bm{u}^i_{\nu,\varepsilon,\mu}\}_{i=1}^m$, such that:, for any $i = 1, 2, 3, \dots$,
  \begin{gather}
      \frac{C_A}{4\tau}|\bm{u}^{i}_{\nu,\varepsilon,\mu}-\bm{u}^{i-1}_{\nu,\varepsilon,\mu}|^2_{[H]^M}+\frac{\mu}{2\tau}|\nabla(\bm{u}^{i}_{\nu,\varepsilon,\mu}-\bm{u}^{i-1}_{\nu,\varepsilon,\mu})|^2_{[H]^M}\nonumber
      \\
      +\mathscr{E}_{\nu,\varepsilon}(\bm{u}^{i}_{\nu,\varepsilon,\mu}) \leq \mathscr{E}_{\nu,\varepsilon}(\bm{u}^{i-1}_{\nu,\varepsilon,\mu}).\label{f-ene0}
  \end{gather}
\end{lemma}
\begin{lemma}\label{lem002}
For any $\bm{w}_0\in [V]^M$, there exists a sequence $ \{\bm{w}_\nu\}_{\nu\in(0,1)}\subset[W^{1,p}(\Omega)]^M $ such that \vspace{-1ex}
\begin{align*}
  \bm{w}_\nu \rightarrow \bm{w}_0 \mbox{ in } [V]^M, \mbox{ and } \mathscr{E}_\nu(\bm{w}_\nu)\rightarrow \mathscr{E}_0(\bm{w}_0) \mbox{ as } \nu\downarrow 0.
\end{align*}
\end{lemma}
\begin{remark}\label{rem001}
    From Lemma \ref{lem002}, there exists a constant $\nu_0\in(0,1)$ such that \vspace{-2ex}
    \begin{align*}
      |\mathscr{E}_\nu(\bm{w}_\nu)-\mathscr{E}_0(\bm{w}_0)|<1 \mbox{ for any $\nu\in(0,\nu_0)$}.
    \end{align*}
\end{remark}
\begin{lemma}\label{lem003}
    For any $\bm{\psi}\in [\mathscr{V}]^M$, and $\{\nu_n\}_{n\in\N}\subset (0,1);\nu_n\downarrow0$, there exists a sequence $\{ {\bm{\psi}_n} \}_{n\in\N} \subset [C^\infty(\overline{Q})]^M$ satisfying the following properties.\vspace{-1ex}
    \begin{align*}
      &\bm{\psi}_n \rightarrow\bm{\psi} \mbox{ in }[\mathscr{V}]^M, \mbox{ and }\nu^{\frac{1}{p}}_n\nabla\bm{\psi}_n \rightarrow 0 \mbox{ in }L^p(0,T;[L^p(\Omega)]^{M\times N})\mbox{ as }n\rightarrow\infty.
    \end{align*}
  \end{lemma}
Let $\{\bm{u}^0_\nu\}_{\nu\in(0,\nu_0)}$ be a sequence of initial value for Lemma \ref{003Thm1}, constructed by Lemma \ref{lem002} with $\bm{w}_0=\bm{u}_0$. Then, owing to the uniform estimate for $ \{ \|\nabla \gamma_\varepsilon\|_{L^\infty} \}_{\varepsilon \in (0, 1)} $ and Remark \ref{rem001}, we have
   \begin{align}\label{constc_E}
       c_E:=\sup_{\substack{\nu\in(0,\nu_0)\\ \varepsilon,\mu\in(0,1)}}\mathscr{E}_{\nu,\varepsilon}(\bm{u}^{0}_{\nu,\varepsilon,\mu}) \leq \mathscr{E}_0(\bm{u}_0)+1+\|\nabla \gamma\|_{L^\infty}|\Omega|<\infty,
   \end{align}
    {where $ |\Omega| $ denotes the Lebesgue measure of $ \Omega \subset \R^N $. }

    Additionally, taking into account Example \ref{ex1}, we observe that
  \begin{align}
    &\nabla\gamma_{\varepsilon}(B( [\underline{\bm{u}}_{\nu,\varepsilon,\mu}]_\tau(t))\nabla [\underline{\bm{u}}_{\nu,\varepsilon,\mu}]_\tau(t))\in\partial\Phi_{\varepsilon}(B( [\underline{\bm{u}}_{\nu,\varepsilon,\mu}]_\tau(t))\nabla [\underline{\bm{u}}_{\nu,\varepsilon,\mu}]_\tau(t))\,
    \label{subdig1}
    \\
    &\hspace{12ex}\mbox{ for all } \nu\in (0,\nu_0),~\varepsilon,\mu\in(0,1), \mbox{ and a.e. }t\in(0,T),
    \nonumber
  \end{align}
  and hence, 
  \begin{gather}
    \nabla\gamma_{\varepsilon}(B( [\underline{\bm{u}}_{\nu,\varepsilon,\mu}]_\tau)\nabla [\underline{\bm{u}}_{\nu,\varepsilon,\mu}]_\tau)\in\partial\widehat{\Phi}^I_{\varepsilon}(B( [\underline{\bm{u}}_{\nu,\varepsilon,\mu}]_\tau)\nabla [\underline{\bm{u}}_{\nu,\varepsilon,\mu}]_\tau),
    \label{subdig2}
    \\
    \mbox{ for all } \nu\in (0,\nu_0),~\varepsilon,\mu\in(0,1), \mbox{ and any open interval }I\subset(0,T),
    \nonumber
  \end{gather}
  where  $[\bm{u}_{\nu,\varepsilon,\mu}]_\tau$, $ [\overline{\bm{u}}_{\nu,\varepsilon,\mu}]_\tau $, and $ [\underline{\bm{u}}_{\nu,\varepsilon,\mu}]_\tau $ are the time-interpolation of $ \{ \bm{u}_{\nu,\varepsilon,\mu}^i \}_{i = 1}^m $, for $ \nu\in(0,\nu_0),~\varepsilon \in (0, 1),~\mu \in (0, 1) $ and $ \tau \in (0, \tau_0(\nu,\varepsilon,\mu)) $.

  \noindent
  {\it Proof of Theorem 1.}  From Lemma \ref{003Thm1} and \eqref{upsip}, we can see the following boundedness:
\begin{description}
    \item[(B-1)]$\{[\bm{u}_{\nu,\varepsilon,\mu}]_\tau~|~\tau\in(0,\tau_0),~\nu\in(0,\nu_0),~\varepsilon\in(0,1),\mu\in(0,1)\}$ is bounded in $L^\infty(0,T;[V]^M)$ and in $W^{1,2}(0,T;[H]^M)$,\vspace{-1ex}
    \item[(B-2)] $\{\sqrt{\mu}\partial_t[\nabla\bm{u}_{\nu,\varepsilon,\mu}]_\tau~|~\tau\in(0,\tau_0),~\nu\in(0,\nu_0),~\varepsilon\in(0,1),\mu\in(0,1)\}$ is bounded in $L^2(0,T;[H]^{M\times N})$,\vspace{-1ex}
    \item[(B-3)]$\{[\overline{\bm{u}}_{\nu,\varepsilon,\mu}]_\tau~|~\tau\in(0,\tau_0),~\nu\in(0,\nu_0),~\varepsilon\in(0,1),\mu\in(0,1)\}$ and $\{[\underline{\bm{u}}_{\nu,\varepsilon,\mu}]_{\tau}~|~\tau\in(0,\tau_0),~\nu\in(0,\nu_0),~\varepsilon\in(0,1),\mu\in(0,1)\}$ are bounded in $L^\infty(0,T;[V]^M)$,\vspace{-1ex}
    \item[(B-4)]$\{\nabla\gamma_{\varepsilon}(B( [\underline{\bm{u}}_{\nu,\varepsilon,\mu}]_\tau(t))\nabla [\underline{\bm{u}}_{\nu,\varepsilon,\mu}]_\tau(t))~|~\tau\in(0,\tau_0),~\nu\in(0,\nu_0),~\varepsilon\in(0,1),\mu\in(0,1)\}$ is bounded in $L^\infty(Q;\R^{M\times N})$,\vspace{-1ex}
    \item[(B-5)]The function of time $t\in[0,T]\mapsto \mathscr{E}_{\nu,\varepsilon}([\overline{\bm{u}}_{\nu,\varepsilon,\mu}]_\tau(t))\in[0,\infty)$ and $t\in [0,T]\mapsto \mathscr{E}_{\nu,\varepsilon}([\underline{\bm{u}}_{\nu,\varepsilon,\mu}]_\tau(t))\in[0,\infty)$ are nonincreasing for every $\nu\in (0,\nu_0)$, $0<\varepsilon<1$, $\mu\in(0,1)$ and $0<\tau<\tau_0(\varepsilon)$. Moreover, $\{\mathscr{E}_{\nu,\varepsilon}(\bm{u}_{\nu}^0)~|~\nu\in(0,\nu_0),~\varepsilon\in(0,1)\}$ is bounded, and hence, the class $\{\mathscr{E}_{\nu,\varepsilon}([\overline{\bm{u}}_{\nu,\varepsilon,\mu}]_{\tau})~|~\tau\in(0,\tau_0),~\nu\in(0,\nu_0),~\varepsilon\in(0,1),\mu\in(0,1)\} $ and $\{\mathscr{E}_{\nu,\varepsilon}([\underline{\bm{u}}_{\nu,\varepsilon,\mu}]_{\tau})~|~\tau\in(0,\tau_0),~\nu\in(0,\nu_0),~\varepsilon\in(0,1),\mu\in(0,1)\} $ are bounded in $BV(0,T)$,\vspace{-1ex}
    \item[(B-6)] $\{\nu\Upsilon_p(\nabla[\overline{\bm{u}}_{\nu,\varepsilon}]_\tau)~|~\tau\in(0,\tau_0),~\nu\in(0,\nu_0),~\varepsilon\in(0,1),\mu\in(0,1)\}$ and $\{\nu\Upsilon_p(\nabla[\underline{\bm{u}}_{\nu,\varepsilon}]_{\tau})~|~\tau\in(0,\tau_0),~\nu\in(0,\nu_0),~\varepsilon\in(0,1),\mu\in(0,1)\}$ are bounded in $L^\infty(0,T;L^1(\Omega))$.
  \end{description}

  On account of (B-1)--(B-3), we can apply the compactness theory of Aubin's type \cite[Corollary 4]{MR0916688}, and find a sequence $\{\nu_n\}_{n\in\N}\subset (0,\nu_0)$, $\{\varepsilon_n\}_{n\in\N}\subset(0,1)$ and $\{\mu_n\}_{n\in\N}\subset(0,1)$ with $ \{ \tau_n \}_{n\in\N} \subset (0, 1) $, and a function $\bm{u}\in[\mathscr{H}]^M$ with $\bm{w}^*\in L^\infty(Q;\R^{M\times N})$ such that:\vspace{-1ex} 
\[
    \nu_n \downarrow0,~\varepsilon_n \downarrow 0,~\mu_n \downarrow 0, ~ \tau_n := \frac{1}{2} \bigl( \tau_0(\nu_n,\varepsilon_n,\mu_n) \wedge \nu_n \wedge \varepsilon_n\wedge\mu_n \wedge 1 \bigr) \downarrow 0,\vspace{-2ex}
\]
 \begin{align}
    & \bm{u}_n : =  [{\bm{u}}_{\nu_n,\varepsilon_n,\mu_n}]_{\tau_n}\rightarrow \bm{u}\mbox{ in }C([0,T];[H]^M),\mbox{ weakly in } W^{ 1 , 2 }( 0,T ; [H]^M),
    \nonumber 
    \\
      &\hspace*{10ex}\mbox{ weakly-}* \mbox{ in } L^\infty( 0,T ; [V]^M) , \mbox{ as $ n \to \infty $,} \label{conv_111}
      \\
      &\mu_n\partial_t\nabla\bm{u}_n\rightarrow 0 \mbox{ weakly in }L^2(0,T;[H]^{M\times N}), \mbox{ as $ n \to \infty $,}
    \label{conv_1}
  \end{align}
  and
  \begin{align}
      &\nabla\gamma_{\varepsilon_n}(B( [\underline{\bm{u}}_{\nu_n,\varepsilon_n,\mu_n}]_{\tau_n})\nabla [\underline{\bm{u}}_{\nu_n,\varepsilon_n,\mu_n}]_{\tau_n})~\rightarrow\bm{w}^*\mbox{ weakly-}*\mbox{in }L^\infty(Q),
    \label{conv_2}
  \end{align}
  in particular, from Lemma \ref{lem002} 
  \begin{align}\label{conv_3}
    \bm{u}(0)=\lim_{n\rightarrow\infty}\bm{u}_n(0)=\lim_{n\rightarrow\infty}\bm{u}^0_{\nu_n}=\bm{u}_0\mbox{ in }[H]^M.
  \end{align}
    Here, since, $\mbox{ for }t\in[t_{i-1},t_i),~i=1,\dots,m,~\tau\in(0,\tau_0),~\nu,\varepsilon,\mu\in(0,1),$
  \begin{align*}
    &\left\{
    \begin{aligned}
      &\max\{|([\overline{\bm{u}}_{\nu,\varepsilon,\mu}]_{\tau}-\bm{u}_{n})(t)|_{[H]^M},|([\underline{\bm{u}}_{\nu,\varepsilon,\mu}]_{\tau}-\bm{u}_{n})(t)|_{[H]^M}\} 
      \\
      &\quad\leq \int_{t_{i-1}}^{t_i}|\partial_t\bm{u}_{n}(t)|_{[H]^M}\,dt\leq \tau^{\frac{1}{2}}|\partial_t\bm{u}_{n}|_{[\mathscr{H}]^M},
    \end{aligned}
    \right.
    \end{align*}
  one can also see from \eqref{conv_111} that:
  \begin{align}
    &\overline{\bm{u}}_{n}:=[\overline{\bm{u}}_{\nu_n,\varepsilon_n,\mu_n}]_{\tau_n}\rightarrow \bm{u},~\underline{\bm{u}}_{n}:=[\underline{\bm{u}}_{\nu_n,\varepsilon_n,\mu_n}]_{\tau_n}\rightarrow \bm{u} \mbox{ in }L^\infty(0,T;[H]^M),
    \nonumber
    \\
    &\hspace*{10ex}\mbox{ weakly-}* \mbox{ in }L^\infty(0,T;[V]^M),\mbox{ as } n\rightarrow\infty,
    \label{conv_4}
    \\
    &\overline{\bm{u}}_{n}(t)\rightarrow \bm{u}(t), \underline{\bm{u}}_{n}(t)\rightarrow \bm{u}(t)\mbox{ in }[H]^M \mbox{ and weakly in }[V]^M,
    \nonumber
    \\
    &\hspace{15ex}\mbox{ as } n\rightarrow\infty,\mbox{ for any }t\in(0,T).
    \label{conv_5}
\end{align}
Moreover, (B-4) enable us to see
\begin{gather}\label{conv_6}
    \bigl| \mathscr{E}_{\nu_n,\varepsilon_n}(\overline{\bm{u}}_n) -\mathscr{E}_{\nu_n,\varepsilon_n}(\underline{\bm{u}}_n) \bigr|_{L^1(0, T)} \leq 2c_E \tau_n \to 0, \mbox{ as $ n \to \infty $}.
\end{gather}
So, applying Helly's selection theorem \cite[Chapter 7, p.167]{rudin1976principles}, we will find a bounded and nonincreasing function $\mathcal{J}_*:[0,T]\mapsto[0,\infty)$, such that 
\begin{align}
    &\mathscr{E}_{\nu_n,\varepsilon_n}(\overline{\bm{u}}_{n})\rightarrow\mathcal{J}_* \mbox{ and } \mathscr{E}_{\nu_n,\varepsilon_n}(\underline{\bm{u}}_{n})\rightarrow\mathcal{J}_* \nonumber
  \\
  & \qquad \mbox{ weakly-}*\mbox{ in }BV(0,T),\mbox{ and }\mbox{weakly-}* \mbox{ in }L^\infty(0,T),
  \label{conv_ene05}
  \\
    &\mathscr{E}_{\nu_n,\varepsilon_n} (\overline{\bm{u}}_{n}(t)) \rightarrow \mathcal{J}_*(t) \mbox{ and } \mathscr{E}_{\nu_n,\varepsilon_n} (\underline{\bm{u}}_{n}(t)) \rightarrow \mathcal{J}_*(t), \mbox{ for any }t\in[0,T], \nonumber
\end{align}
as $ n \to \infty $, by taking a subsequence if necessary.
Now, let us show the limit function $\bm{u}$ is a solution to the system (S). (S0) can be checked by \eqref{conv_111} and \eqref{conv_3}. Next, let us show $\bm{u}$ satisfies the variational inequalities (S1). Let us take any $ t \in (0, T] $. Then, from \eqref{3TimeDis-02}, the sequences as in \eqref{conv_1}--\eqref{conv_4} satisfy the following inequality:
      \begin{align}
    &\int^t_0(A(\underline{\bm{u}}_n(\sigma))\partial_t\bm{u}_n,(\overline{\bm{u}}_n-\bm{\omega})(\sigma))_{[H]^M}\,d\sigma+\nu_n\int^t_0\int_\Omega \Upsilon_p(\nabla \overline{\bm{u}}_n(\sigma))\,dxd\sigma\nonumber
    \\
    &\quad+\mu_n\int^t_0(\nabla\partial_t\bm{u}_n(\sigma),\nabla(\overline{\bm{u}}_n-\bm{\omega})(\sigma))_{[H]^{M\times N}}\,d\sigma
    \nonumber
    \\
    &\quad+\int^t_0(\nabla_{\bm{u}} G(x,\overline{\bm{u}}_n(\sigma)),(\overline{\bm{u}}_n-\bm{\omega})(\sigma))_{[H]^M}\,d\sigma
    \nonumber
    \\
    &\quad +\int^t_0([\nabla\alpha](\underline{\bm{u}}_n(\sigma))\gamma_{\varepsilon_n}(B(\underline{\bm{u}}_n(\sigma))\nabla\underline{\bm{u}}_n(\sigma)),(\overline{\bm{u}}_n-\bm{\omega})(\sigma))_{[H]^M}\,d\sigma
    \nonumber
    \\
    &\quad+\frac{\kappa}{2}\int_0^t |\nabla\overline{\bm{u}}_n(\sigma)|^2_{[H]^{M\times N}}\,d\sigma+\int^t_0(\alpha(\underline{\bm{u}}_n(\sigma))\nabla\gamma_{\varepsilon_n}({B}(\underline{\bm{u}}_n(\sigma))\nabla\underline{\bm{u}}_n(\sigma)):\nonumber
    \\
    &\qquad:[\nabla{B}](\underline{\bm{u}}_n(\sigma))\nabla\underline{\bm{u}}_n(\sigma),(\overline{\bm{u}}_n-\bm{\omega})(\sigma))_{[H]^M}\,d\sigma
\nonumber
    \\
    &\quad +\int^t_0\int_\Omega\alpha(\overline{\bm{u}}_n(\sigma))\gamma_{\varepsilon_n}({B}(\overline{\bm{u}}_n(\sigma))\nabla\overline{\bm{u}}_n(\sigma))\,dxd\sigma
    \label{conv_8}
    \\
    &\leq \int^t_0\int_\Omega\alpha(\overline{\bm{u}}_n(\sigma))\gamma_{\varepsilon_n}({B}(\overline{\bm{u}}_n(\sigma))\nabla\bm{\omega}(\sigma))\,dxd\sigma+\nu_n\int^t_0\int_\Omega \Upsilon_p(\nabla \bm{\omega}(\sigma))\,dxd\sigma
    \nonumber
    \\
    &\quad+\frac{\kappa}{2}\int_0^t |\nabla\bm{\omega}(\sigma)|^2_{[H]^{M\times N}}\,d\sigma\mbox{ for all } \bm{\omega}\in L^2(0,T;[W^{1,p}(\Omega)]^M), \mbox{ and }n\in\N.
    \nonumber
  \end{align}
 From \eqref{conv_5}, (B-6), and the fact that $\Upsilon_p$ takes nonnegative values, the following is obvious.
  \begin{align}\label{conv_9}
    &\ulim_{n \to \infty}\int_0^t\int_\Omega|\nabla\bm{u}_n(\sigma)|^2_{[H]^{M\times N}}\,dxd\sigma\geq\int_0^t\int_\Omega|\nabla\bm{u}(\sigma)|^2_{[H]^{M\times N}}\,dxd\sigma,
    \\
    &\ulim_{n \to \infty}\nu_n\int_0^t\int_\Omega  \Upsilon_p(\nabla\overline{\bm{u}}_n(\sigma))\,dxd\sigma\geq0.
  \end{align}
  Also, by using (A1), (A3), (A5), \eqref{gamma_ep} and \eqref{conv_4}, weakly lower-semicontinuity of $\gamma_\varepsilon$ and Fatou's lemma, one can see 
  \begin{align}
    &\ulim_{n\rightarrow\infty} \int^t_0\int_\Omega\alpha(\overline{\bm{u}}_n(\sigma))\gamma_{\varepsilon_n}({B}(\overline{\bm{u}}_n(\sigma))\nabla\overline{\bm{u}}_n(\sigma))\,dxd\sigma
    \nonumber
    \\
    &\geq -C_\gamma\|\nabla\alpha\|_{L^\infty}\|B\|_{L^\infty}\lim_{n\rightarrow\infty}\int^t_0\int_\Omega|\overline{\bm{u}}_n(\sigma)-{\bm{u}}(\sigma)|(|\nabla\overline{\bm{u}}_n(\sigma)|+1)\,dxd\sigma
    \nonumber
    \\
    &\quad -\lim_{n\rightarrow\infty}\|\gamma_{\varepsilon_n}-\gamma\|_{L^\infty}\int^t_0\int_\Omega\alpha(\bm{u}(\sigma))\,dxd\sigma
    \nonumber
    \\
    &\quad +\liminf_{n\rightarrow\infty}\int^t_0\int_\Omega\alpha(\bm{u}(\sigma))\gamma({B}(\overline{\bm{u}}_n(\sigma))\nabla\overline{\bm{u}}_n(\sigma))\,dxd\sigma
    \nonumber
    \\
    &\geq \int^t_0\int_\Omega\alpha(\bm{u}(\sigma))\gamma({B}({\bm{u}}(\sigma))\nabla{\bm{u}}(\sigma))\,dxd\sigma.\label{conv_10}
  \end{align}
  Let $\{\bm{\psi}_n\}_{n\in\N}$ be the sequence as in Lemma \ref{lem003} corresponding to the case $\bm{\psi}=\bm{u}$. Then, substituting $\bm{\psi}_n$ for $\bm{\omega}$ in \eqref{conv_8}, we see that
  \begin{align}
    &\olim_{n \to \infty}\left(+\int^t_0\int_\Omega\alpha(\overline{\bm{u}}_n(\sigma))\gamma_{\varepsilon_n}({B}(\overline{\bm{u}}_n(\sigma))\nabla\overline{\bm{u}}_n(\sigma))\,dxd\sigma\right.
    \nonumber
    \\
    &\left.+\nu_n\int_0^t\int_\Omega\Upsilon_p(\nabla\overline{\bm{u}}_n(\sigma))\,dxd\sigma+\int_0^t|\nabla \bm{u}_n(\sigma)|_{[H]^{M\times N}}^2\,d\sigma
    \nonumber
    \right)
  \\
    &\leq \lim_{n\rightarrow\infty}\nu_nC_\Upsilon\Big(\int_0^t|\nabla\bm{\psi}_n(\sigma)|^p_{L^p(\Omega)}\,d\sigma+|\Omega|T\Big)
    \nonumber
    \\
    &+\lim_{n\rightarrow\infty}\int_0^t(\mu_n\nabla\partial_t\bm{u}_n(\sigma),\nabla(\overline{\bm{u}}_n-{\bm{\psi}}_n)(\sigma))_{[H]^{M\times N}}\,d\sigma 
    \nonumber
    \\
    &-\lim_{n\rightarrow\infty}\int_0^t(A(\underline{\bm{u}}_n(\sigma))\partial_t\bm{u}_n(\sigma)+\nabla_{\bm{u}} G(x,\overline{\bm{u}}_n(\sigma)),(\overline{\bm{u}}_n-\bm{\psi}_n)(\sigma))_{[H]^M}\,d\sigma
    \nonumber
    \\
    &-\lim_{n\rightarrow\infty}\int^t_0([\nabla\alpha](\underline{\bm{u}}_n(\sigma))\gamma_{\varepsilon_n}(B(\underline{\bm{u}}_n(\sigma))\nabla\underline{\bm{u}}_n(\sigma)),(\overline{\bm{u}}_n-\bm{\psi}_n)(\sigma))_{[H]^M}\,d\sigma\nonumber
    \\
    &+\lim_{n\rightarrow\infty}\int_0^t|\nabla \bm{\psi}_n(\sigma)|_{[H]^{M\times N}}^2\,d\sigma-\lim_{n\rightarrow\infty}\int^t_0(\alpha(\underline{\bm{u}}_n(\sigma))\cdot
    \label{conv_12}
    \\
    &\quad\cdot \nabla\gamma_{\varepsilon_n}({B}_0(\underline{\bm{u}}_n(\sigma))\nabla\underline{\bm{u}}_n(\sigma){B}_1(\underline{\bm{u}}_n(\sigma))):\big([\nabla B_0](\underline{\bm{u}}_n)\nabla\underline{\bm{u}}_n(\sigma) B_1(\underline{\bm{u}}_n) 
    \nonumber
    \\
    &\qquad+ B_0(\underline{\bm{u}}_n)\nabla\underline{\bm{u}}_n(\sigma) [\nabla B_1](\underline{\bm{u}}_n)\big),(\overline{\bm{u}}_n-\bm{\psi}_n)(\sigma))_{[H]^M}d\sigma
    \nonumber
    \end{align}
  \begin{align}
    &=0+\int^t_0\int_\Omega\alpha({\bm{u}}(\sigma))\gamma({B}({\bm{u}}(\sigma)\nabla{\bm{u}}(\sigma)))\,dxd\sigma+\int_0^t|\nabla \bm{u}(\sigma)|_{[H]^{M\times N}}^2\,d\sigma.
    \nonumber
  \end{align}
  From (Fact 1) and \eqref{conv_9}--\eqref{conv_12}, we obtain the following convergences as $n\rightarrow\infty$:
  \begin{align}
    &\nu_n\int_0^t\int_\Omega\Upsilon_p(\nabla\overline{\bm{u}}_n(\sigma))\,dxd\sigma\rightarrow0,\label{conv_13}
    \\
    &\int_0^t|\nabla\overline{\bm{u}}_n(\sigma)|^2_{[H]^{M\times N}}\,d\sigma\rightarrow\int_0^t|\nabla\bm{u}(\sigma)|^2_{[H]^{M\times N}}\,d\sigma.\label{conv_15}
  \end{align}
  Additionally, thanks to \eqref{gamma_ep}, \eqref{conv_15}, (B-4), and uniform convexity of $L^2$-based topologies, we can derive that, for a.e. $t\in[0,T]$, as $n\rightarrow\infty$:
  \begin{align}
    &\bullet\overline{\bm{u}}_n(t)\rightarrow \bm{u}(t)\mbox{ and }\underline{\bm{u}}_n(t)\rightarrow \bm{u}(t)\mbox{ in }[V]^M,\label{conv_16}
    \\[1.0ex]
    &\bullet|\alpha(\overline{\bm{u}}_n(t))\gamma_{\varepsilon_n}({B}(\overline{\bm{u}}_n(t))\nabla\overline{\bm{u}}_n(t))-\alpha({\bm{u}}(t))\gamma({B}({\bm{u}}(t))\nabla{\bm{u}}(t))|_{L^1(\Omega)}
\rightarrow0,\label{conv_17}
  \end{align}
  \eqref{conv_5}--\eqref{conv_ene05}, \eqref{conv_16} and \eqref{conv_17} imply 
  \begin{gather}\label{convJ}
     \begin{cases}
         E_{\nu_n,\varepsilon_n}(\overline{\bm{u}}_n(t)) \to \mathcal{J}_*(t) = E(\bm{u}(t)),
         \\
         E_{\nu_n,\varepsilon_n}(\underline{\bm{u}}_n(t)) \to \mathcal{J}_*(t) = E(\bm{u}(t)), 
     \end{cases}
     \mbox{a.e. $ t \in (0, T) $, as $ n \to \infty $.}
 \end{gather}
 Now, we take $\omega=\bm{\varphi}$ in $[V]^M$ in \eqref{conv_8}, and consider to pass to the limit $n\rightarrow\infty$. Then, in light of \eqref{conv_111}, \eqref{conv_4}, \eqref{conv_13} and \eqref{conv_15}, and the Lebesgue Bochner dominated convergence theorem, we obtain
    \begin{gather}
    \int_I(A(\bm{u}(t))\partial_t\bm{u}(t),\bm{u}(t)-\bm{\varphi})_{[H]^M}\,dt+\kappa\int_I(\nabla\bm{u}(t),\nabla(\bm{u}(t)-\bm{\varphi}))_{[H]^{M\times N}}\,dt
    \nonumber
    \\
    +\int_I(\nabla_{\bm{u}} G(x,\bm{u}(t))+[\nabla\alpha](\bm{u}(t))\gamma({B}(\bm{u}(t))\nabla\bm{u}(t)),\bm{u}(t)-\bm{\varphi})_{[H]^M}\,dt
    \nonumber
    \\
    +\int_I(\alpha(\bm{u}(t))\bm{w}^*(t):[\nabla{B}](\bm{u}(t))\nabla\bm{u}(t),\bm{u}(t)-\bm{\varphi})_{[H]^M}\,dt\label{conv_18}
    \\
    +\int_I\int_{\Omega}\alpha(\bm{u}(t))\gamma(B(\bm{u}(t))\nabla \bm{u}(t))\,dx\,dt\leq\int_I\int_{\Omega}\alpha(\bm{u}(t))\gamma (B(\bm{u}(t))\nabla\bm{\varphi})\,dx\,dt,
    \nonumber
    \end{gather}
    for any open interval $I\subset(0,T)$. Moreover, by involving \eqref{subdig1}, \eqref{subdig2}, \eqref{conv_2}, together with Example \ref{ex2}, and (Fact 3) in Remark \ref{rem2}, one can conclude that 
    \begin{equation*}
     \bm{w}^*\in\partial\widehat{\Phi}_0^I(B(\bm{u})\nabla \bm{u})\mbox{ in }L^2(I;[H]^{M\times N}),
    \end{equation*}
 and hence, 
 \begin{equation}\label{subdig3}
  \begin{aligned}
    &\bm{w}^*\in\partial\Phi_0(B(\bm{u})\nabla \bm{u}) \mbox{ in } [H]^{M\times N}, \mbox{ for a.e. }t\in(0,T),\mbox{ and }
    \\
    &\bm{w}^*\in\partial\gamma(B(\bm{u})\nabla \bm{u}) \mbox{ in }\R^{M\times N}, \mbox{ a.e. in } Q.
  \end{aligned}
 \end{equation}
\eqref{conv_18} and \eqref{subdig3} imply that the limit $\bm{u}$ satisfies (S1).

Next, we verify (S2). From \eqref{f-ene0}, it is derived that 
\begin{align}
      &\int_{t_{i-1}}^{t_i}\biggl(\frac{C_A}{4}|\partial_t\bm{u}_n(\sigma)|^2_{[H]^M}+\frac{\mu_n}{2}|\nabla\partial_t\bm{u}_n(\sigma)|^2_{[H]^{M\times N}}\biggr)\,d\sigma+\mathscr{E}_{\nu_n,\varepsilon_n}(\overline{\bm{u}}_n(t))
      \nonumber
      \\
      &\quad\leq \mathscr{E}_{\nu_n,\varepsilon_n}(\underline{\bm{u}}_n(t)),\mbox{ for }t\in[t_{i-1},t_i),~i=1,2,\dots,{\textstyle\frac{T}{\tau_n}},\mbox{ and }n\in\N.
      \label{energy1}
\end{align}
Here, setting $m^s:=[\frac{s}{\tau}]$ and $m_t:=\bigl([\frac{t}{\tau}]+1\bigr)\wedge \frac{T}{\tau}$ for $0\leq s< t\leq T$, and summing both sides of \eqref{energy1} for $i=m^s+1, m^s+2,\dots,m_t$, we obtain that
\begin{align}
      &\frac{C_A}{4}\int_s^t |\partial_t\bm{u}_n(\sigma)|^2_{[H]^M}\,d\sigma+\mathscr{E}_{\nu_n,\varepsilon_n}(\overline{\bm{u}}_n(t))
      \nonumber
      \\
      &\leq \frac{C_A}{4}\int_{m^s\tau_n}^{m_t\tau_n} \left(|\partial_t\bm{u}_n(\sigma)|^2_{[H]^M}+2\mu_n|\nabla\partial_t\bm{u}_n(\sigma)|^2_{[H]^{M\times N}}\right)\,d\sigma+\mathscr{E}_{\nu_n,\varepsilon_n}(\overline{\bm{u}}_n(t))
      \nonumber
      \\
      &\leq \mathscr{E}_{\nu_n,\varepsilon_n}(\underline{\bm{u}}_n(s)), \mbox{ for }s,t\in[0,T];s\leq t,\mbox{ and }n\in\N.\label{energy3}
    \end{align}
    Now, taking the limit $n\rightarrow\infty$ and using \eqref{conv_111} \eqref{conv_2}, \eqref{convJ} and \eqref{energy3}, we see that 
\begin{gather}
    \frac{C_A}{4}\int_{s}^{t}|\partial_t\bm{u}(\sigma)|^2_{[H]^M}\,d\sigma+\mathscr{E}(\bm{u}(t)) \leq \mathscr{E}(\bm{u}(s)),
    \label{energy4}
    \\
    \mbox{ for a.e. $ s \in [0, T) $ including $s = 0$, and a.e. $ t \in (s, T) $.}
    \nonumber
\end{gather}
In addition, the condition ``a.e. $ t \in (s, T) $'' in \eqref{energy4} can be strengthened to ``for any $ t \in [s, T] $''. Indeed, by taking a sequence $\{t_n\}_{n\in\N}\subset (t,T)$ with $t_n \rightarrow t \in [s, T]$, we observe that 
\begin{gather}
  \frac{C_A}{4}\int_{s}^{t_n}|\partial_t\bm{u}(\sigma)|^2_{[H]^M}\,d\sigma+\mathscr{E}(\bm{u}(t_n)) \leq \mathscr{E}(\bm{u}(s)),\mbox{ for all $ n\in\N $.}
  \label{energy5}
\end{gather}
Having in mind the lower semi-continuity of $\mathscr{E}$ on $[H]^M$ and (S0), we conclude (S2) by verifying that \eqref{ene-inq1} holds for a.e. $s\in (0,T)$ including $s=0$, and any $t\in[s,T]$. 

Thus, we complete the proof of Theorem \ref{mainThm1}. $\square$

\begin{remark}
  For the parabolic system associated with a regularized energy $\mathscr{E}_{\nu}:\bm{u}\in[W^{1,p}(\Omega)]^M \mapsto \mathscr{E}_\nu(\bm{u}):=\mathscr{E}(\bm{u})+\nu\int_\Omega \Upsilon_p(\nabla \bm{u})\,dx$, with $\nu>0$ and $p\in(2,\infty)\cap[N,\infty)$, we obtain results analogous to Theorem \ref{mainThm1}. Moreover, when $A=I$ (identity) and $\gamma \in C^{1,1}(\mathbb{R}^{M \times N})$, the results of uniqueness and continuous dependence of solutions follow from arguments similar to those in \cite{arAMSU2025}. 

\end{remark}

\noindent
{\bf{Acknowledgements.}} This work is supported by JST SPRING Grant Number JP-
MJSP2109.

\end{document}